\documentclass[11pt,A4paper]{article}

\usepackage{amsmath,amssymb,amsthm,amsfonts}
\usepackage{enumerate}
\usepackage[english]{babel}
\usepackage[latin1]{inputenc}
\usepackage{color}

\newtheorem{theorem}{Theorem}[section]
\newtheorem{proposition}[theorem]{Proposition}

\newtheorem{lemma}[theorem]{Lemma}

\newtheorem{corollary}[theorem]{Corollary}

\newcommand{\R}{\mathbb{R}}

\newenvironment{demT2}{\textit{Proof of Theorem~\ref{T2}.}}{\hfill$\square$}

\usepackage{color,enumitem,graphicx}
\usepackage[colorlinks=true,urlcolor=blue,
citecolor=red,linkcolor=blue,linktocpage,pdfpagelabels,
bookmarksnumbered,bookmarksopen]{hyperref}
\usepackage[english]{babel}
\usepackage{enumitem}

\begin{document}
\title{On the structure of positive solutions for a class of quasilinear  equations\thanks{Research partially supported by CAPES and CNPq grants 308735/2016-1 and 307770/2015-0\newline E-mail addresses: {\tt
everaldo@mat.ufpb.br} (Everaldo Medeiros), {\tt
uberlandio@mat.ufpb.br} (Uberlandio Severo), {\tt
willian\_matematica@hotmail.com} (Willian Cintra) }}

\author{{Willian Cintra,\, Everaldo Medeiros\,\  and\,\  Uberlandio Severo}\\
{\small Universidade Federal da Para\'\i ba,\, Departamento de Matem\'atica}\\
{\small 58051-900, Jo\~ao Pessoa--PB,
Brazil}}
\maketitle

\begin{abstract}
This paper studies the existence, nonexistence and uniqueness of positive solutions for a class of quasilinear equations. We also analyze the behavior of this solutions with respect to two parameters $\kappa$ and $\lambda$ that appears in the equation. The proof of our main results relies on bifurcation techniques, the sub and supersolution method and a construction of an appropriate large solutions.

\bigskip


\end{abstract}

\section{Introduction}
The main goal of this paper is to study existence, nonexistence, uniqueness and asymptotic behavior of positive solutions for the following class of quasilinear equations
\begin{equation}\label{P}
    \left\{ \begin{aligned}
         -\Delta u -\kappa \Delta(u^2)u &= \lambda u - b(x) |u|^{p-1}u &\mbox{in}&\quad\Omega,\\
         u&=0 & \mbox{on}&\quad\partial \Omega,
    \end{aligned} \right.
    \tag{$\mathcal{P_\kappa}$}
\end{equation}
where $\Omega\subset\mathbb{R}^N$ $(N\geq1)$ is a smooth bounded domain, $p>1$ is a constant, $\kappa$ and $\lambda$ are positive parameters and the weight function $b(x)$ satisfies certain regularity conditions.

Problem \eqref{P} with $\kappa=0$ becomes the classical semilinear elliptic problem
\begin{equation}\label{Plog}
\left\{ \begin{aligned}
-\Delta u  &= \lambda u - b(x) |u|^{p-1}u &\mbox{in}&\quad\Omega,\\
u&=0 & \mbox{on}&\quad\partial \Omega
\end{aligned} \right.
\tag{$\mathcal{P}_0$}
\end{equation}
whose positive solutions are \textit{equilibria} or \textit{stationary} solutions of the the following reaction diffusion problem of logistic type

\begin{equation}
\left\{ \begin{aligned}
u_t-\Delta u &= \lambda u - b(x) u^p &\mbox{in}&\quad\Omega,\quad t>0\\
u&=0 & \mbox{on}&\quad\partial \Omega,\quad t>0\\
u(0)&=u_0\geq0,
\end{aligned} \right.
\end{equation}
see for instance \cite{Arrieta} and references therein. Problem \eqref{Plog} has been  object of intense study  by many authors. If $\lambda\leq\lambda_1$ ($\lambda_1$ being the principal eigenvalue of $(-\Delta, H^1_0(\Omega))$ then the problem \eqref{Plog} could have only the trivial solution, see Ambrosetti \cite{Ambrosetti}. In \cite{Ambrosetti-Mancini}  Ambrosetti-Macini  prove that if $\lambda>\lambda_1$ then the problem has two nontrivial solutions of constant sign (one positive and the other negative). Soon thereafter Struwe \cite{Struwe} improved the result and proved that if $\lambda>\lambda_2$ the problem \eqref{Plog} has three nontrivial solutions. Subsequently Ambrosetti-Lupo \cite{Ambrosetti-Lupo} slightly improved the work of Struwe \cite{Struwe} and also presented an approach based on Morse theory. This class of problems involving a more general operator as the $p-laplacian$ (see for instance \cite{Papageorgiou}). In \cite{Olimpio2016} the authors have consider problem \eqref{P} with $b=0$ and they proved that $\eqref{P}$ has only the trivial solution if $\lambda<\lambda_1$.

The study of quasilinear equations involving the operator $L_{\kappa}u:=\Delta u+\kappa\Delta(u^2)u$ arises in various branches of mathematical physics. It is well-known that nonlinear Schr\"{o}dinger equations of the form
\begin{equation}\label{SCH2}
i\partial_t \psi =-\Delta \psi+V(x)\psi- \kappa\Delta
[|\psi|^2]\psi-h(|\psi|^2)\psi,
\end{equation}
where $\psi:\mathbb{R}\times  \mathbb{R}^N \rightarrow\mathbb{C}$,
$V=V(x)$ is a given potential, $\kappa$ is a real constant and
$h$ is a real function, have been studied in relation with some
mathematical models in physics (see for instance \cite{Poppenberg-Schmitt-Wang}).
It was shown that a system describing the self-trapped electron on
a lattice can be reduced in the continuum limit to (\ref{SCH2})
and numerics results on this equation are obtained in
\cite{Brizhik}. In \cite{Hartmann}, motivated by the nanotubes and
fullerene related structures, it was proposed and shown that a
discrete system describing the interaction of a 2-dimensional
hexagonal lattice with an excitation caused by an excess electron
can be reduced to (\ref{SCH2}) and numerics results have been done
on domains of disc type, cylinder type and sphere type.

Setting $\psi(t,x)=\exp(-iFt)u(x),\; F\in \mathbb{R}$, into the
equation (\ref{SCH2}), we obtain a corresponding equation
\begin{equation}\label{simp}
-\Delta u -\kappa\Delta (u^2)u = g(u)-V(x)u\quad\mbox{in} \quad \Omega,\\
\end{equation}
where we have renamed $V(x)-F$ to be $V(x)$ and
$g(u)=h(u^2)u$.

The quasilinear equation (\ref{simp}) in the whole
$\mathbb{R}^{N}$ has received special attention in the past
several years, see for instance \cite{Jeanjean2014, do O-Severo, Liuwang, Poppenberg-Schmitt-Wang} and
references therein. In these papers, we find important results on
the existence of nontrivial solutions of (\ref{simp}) and a good
insight into this quasilinear Schr\"{o}dinger equation. The main
strategies used are the following: the first of them consists in
by using a constrained minimization argument, which gives a
solution of (\ref{simp}) with an unknown Lagrange multiplier
$\lambda$ in front of the nonlinear term, see for example
\cite{Poppenberg-Schmitt-Wang}. The other one consists in using
a change of variables to get a new semilinear equation and
an appropriate Orlicz space framework, for more details see
\cite{Jeanjean2014, do O-Severo,  Liuwang}. In general, the existence results for  equations of the type \eqref{simp} is obtained by using  variational methods. Here, we intend to use bifurcation techniques and the sub and supersolution method in order to analyze \eqref{P}.

In addition to studies involving the operator $L_{\kappa}u$, another important motivation to study problem \eqref{P} is the fact that many papers have been devoted to study quasilinear and semilinear equations involving logistic terms, which appear naturally
in several contexts. For instance, when $\kappa=0$, problem \eqref{P} becomes the classical logistic equation with linear diffusion and refuge,
where $u(x)$ describes the density of the individuals of  species at the location $x\in\Omega$, the nonlinearity $g(x,u):=\lambda u -b(x)u^p$ is the well-known logistic reaction term. There are several papers available in the literature dedicated to the analysis of \eqref{Plog}, see for example
\cite{CC2003, Metas,Okubo,OkuboL}  and references therein.

It is worth mentioning that \eqref{P} can be seen as a quasilinear perturbation of the classical equation \eqref{Plog}, specially when $\kappa \simeq 0$. As we shall see in Theorem \ref{T1} and \ref{T2}, the presence of this quasilinear term prevents the blow-up \eqref{logInfty} that occurs with the positive solutions of \eqref{Plog}. Moreover, when $\kappa \downarrow 0$ the positive solutions of \eqref{P} tends to the positive solutions of \eqref{Plog}.

In order to study the behavior of positive solutions of problem \eqref{P} with respect to the parameter $\kappa>0$, we will assume the following assumptions on $b(x)$:
\begin{enumerate}
	\item[($b_0$)] The function $b:\overline{\Omega} \rightarrow [0,\infty)$ belongs to $C^{\alpha}(\overline{\Omega})$ for some $0<\alpha<1$;
	\item[($b_1$)] The open set $\Omega_+:=\{x \in \Omega;~b(x)>0\}$
	satisfies $\overline{\Omega}_{+} \subset \Omega$ and there is a finite number of smooth components $\Omega_{+}^{j}$, $j=1,\ldots n$, such that $\overline{\Omega_{+}^j}  \cap \overline{\Omega_{+}^i} = \emptyset$ if $i\neq j$. Moreover, the open set
	$$
	\Omega_{b,0}:= \Omega \setminus \overline{\Omega}_+
	$$
	is connected. It should be noted that $\partial \Omega_+ \subset \Omega$ and $\partial \Omega_{b,0} = \partial\Omega \cup \partial \Omega_+$.
\end{enumerate}


Before to state our main results, in a precise form, let us recall some notations. Throughout, for any function $V \in L^\infty(\Omega)$ called \textit{potential},  by $\lambda_1[-\Delta +V]$
we mean the principal eigenvalue of $-\Delta + V$ in $\Omega$ under homogeneous Dirichlet boundary conditions. By simplicity, we also use the convention
$\lambda_1 : =\lambda_1[-\Delta]$. Moreover, we will denote by  $\lambda_{b,0}$ the principal eigenvalue of $-\Delta$ in $\Omega_{b,0}$ under homogeneous Dirichlet boundary conditions.

Now, we are in position to state our main results on the problem \eqref{P}. 

\begin{theorem}\label{T1}
Let $p>1$, $\kappa>0$ and assume $(b_0)$. Then problem \eqref{P} has a positive solution if and only if $\lambda> \lambda_1$. Moreover, if $p \geq 3$ or $b(x)\equiv b>0$ is a constant,  it is unique if it exists and it will be denoted by $\Psi_{\lambda,\kappa}$. In addition, the map $
\lambda \in (\lambda_1,+\infty) \mapsto \Psi_{\lambda,\kappa}  \in \mathcal{C}_0^1(\overline{\Omega})$
is  increasing, in the sense that $\Psi_{\lambda,\kappa}> \Psi_{\mu,\kappa}$ if $\lambda> \mu > \lambda_1$.
Furthermore,
\begin{equation}\label{liml1}
\lim_{\lambda \downarrow \lambda_1}\|\Psi_{\lambda,\kappa}\|_\infty  =0
\end{equation}
and for any compact $K\subset\overline{\Omega}_{b,0} \setminus \partial\Omega$,
\begin{equation}\label{vinfinito}
    \lim_{\lambda \rightarrow +\infty} \Psi_{\lambda,\kappa} = \infty \quad \mbox{uniformly in }K.
\end{equation}
\end{theorem}
Note that we do not assume the hypothesis $(b_1)$ in this theorem. Moreover, we also analyzed the positive solution of \eqref{P} in the particular case $b(x)\equiv b >0$ constant (see Proposition~\ref{unicidade}). It should be noted that ours assumptions on the weight function $b(x)$ include the case $b \equiv 0$. Thus, Theorem~\ref{T1}  still holds for  $b \equiv0$, improving the results of \cite[Theorem 1.1]{Olimpio2016}.

Concerning the asymptotic behavior of positive solutions $\Psi_{\mu,\kappa}$ in Theorem \ref{T1}, based on the ideas in the paper \cite{DLS2002}, we will analyze  the point-wise behavior of the positive solutions of \eqref{P} with respect to $\kappa\downarrow 0$. To this, let us recall the main result concerning to problem \eqref{Plog} (see for instance Theorem 1.1 in \cite{DLS2002} and references therein).

\begin{theorem}\label{Tlog}
	Assume $(b_0)$, $(b_1)$ and $p>1$.  Then the following assertions hold:
	\begin{description}
		\item[$(a)$] The problem \eqref{Plog} has a positive solution if and only if $\lambda \in
		(\lambda_1, \lambda_{b,0})$. Moreover, it is unique if it exists and it will be denoted by $\Theta_\lambda$.
		In addition, $\Theta_\lambda$
		is a nondegenerate solution of \eqref{Plog} and the map $\lambda \in (\lambda_1,\lambda_{b,0}) \mapsto \Theta_\lambda  \in \mathcal{C}_0^1(\overline{\Omega})$
		is increasing, in the sense that $\Theta_\lambda > \Theta_\mu \quad \mbox{if}\quad \lambda_{b,0}> \lambda> \mu > \lambda_1$.
		Furthermore, for each compact  $K \subset \overline{\Omega}_{b,0}\setminus \partial\Omega$,
		\begin{equation}\label{logInfty}
		\lim_{\lambda \rightarrow \lambda_{b,0}} \Theta_\lambda = \infty \quad \mbox{uniformly in }K
		\end{equation}
		and, for each compact $K \subset \Omega_+$,
		\begin{equation}\label{loglarga}
		\lim_{\lambda \rightarrow \lambda_{b,0}} \Theta_\lambda = M_{\lambda_{b,0}} \quad \mbox{uniformly in }K,
		\end{equation}
		where $M_{\lambda_{b,0}}$ stands for the minimal positive classical solution of the singular boundary value problem
		\begin{equation}\label{Ploglargo}
		\left\{ \begin{aligned}
		-\Delta u  &= \lambda u - b(x) u^p &\mbox{in}&\quad\Omega_+,\\
		u&=\infty & \mbox{on}&\quad\partial \Omega_+,
		\end{aligned} \right.
		\end{equation}
		with $\lambda = \lambda_{b,0}$.
		\item[$(b)$] Problem \eqref{Ploglargo} possesses a minimal positive solution for each $\lambda \in \R$ and it will be denoted by $M_\lambda$.
	\end{description}
\end{theorem}
Concerning the asymptotic behavior of positive solutions $\Psi_{\lambda,\kappa}$ given in Theorem~\ref{T1} we will prove the following result.
\begin{theorem}\label{T2}
Suppose $(b_0)$ and $(b_1)$. The following assertions are true:
\begin{description}
    \item[$(a)$] If $\lambda \in (\lambda_1, \lambda_{b,0})$ then
$\lim_{\kappa \downarrow 0} \Psi_{\lambda,\kappa} = \Theta_\lambda$ in $\mathcal{C}_0^1(\overline{\Omega});$
    \item[$(b)$] If $\lambda \geq \lambda_{b,0}$ then, for any compact $K \subset \overline{\Omega}_{b,0}\setminus \partial\Omega$,
    \begin{equation}\label{k1}
            \lim_{\kappa \downarrow 0} \Psi_{\lambda,\kappa} = +\infty \quad \mbox{uniformly in }K;
    \end{equation}
    \item[$(c)$] Suppose in addiction that $p>3$. If $\lambda \geq \lambda_{b,0}$ then, for any compact $K \subset \Omega_+$
    \begin{equation}\label{k2}
        \lim_{\kappa \downarrow 0} \Psi_{\lambda,\kappa} = M_\lambda \quad \mbox{uniformly in }K,
    \end{equation}
where $M_\lambda$ stands for the minimal positive classical solution of the singular boundary value problem \eqref{Ploglargo}.
\end{description}
\end{theorem}
We point out that the stability of the positive solutions of the dual problem associated to \eqref{P} is also  analyzed (see  Proposition~\ref{estabilidade}).

The outline of this paper is as follows: In Section~\ref{sec2}, we introduce the dual approach of \eqref{P} and prove the first results which play an important role in  our analysis. In Section~\ref{sec3}, we show the existence and uniqueness of positive solutions of \eqref{P}. In Section~\ref{sec4}, we prove a stability result. Section~\ref{sec5} is devoted to prove a pivotal a priori bounds  and in Section~\ref{sec6} we will use theses estimates to study the asymptotic behavior of the positive solution of \eqref{P} when $\kappa \downarrow 0$.

\section{An auxiliary problem}\label{sec2}
In this section we introduce the dual approach developed in the papers \cite{Jeanjean2014,Liuwang} to deal with \eqref{P}. Specifically, we convert the quasilinear equation \eqref{P} into a semilinear one by using a suitable change of variables. To this end, we argue as follows. For each $\kappa\geq 0$, let $f_\kappa:\R \rightarrow \R$ denotes the solution of the Cauchy problem
$$
f_\kappa'(t) = \frac{1}{(1+2\kappa f_\kappa^2(t))^{1/2}}, \quad f_\kappa(0)=0.
$$
By the standard theory of ODE, we obtain that $f_\kappa$ is uniquely determined, invertible and of class $\mathcal{C}^\infty(\mathbb{R},\mathbb{R})$. Moreover, it is well know that the inverse function of $f$ is given by
$$
f_\kappa^{-1}(t) :=\int_0^t (1+2\kappa s^2)^{1/2}ds, \quad \forall\ t \geq 0.
$$
Thus, by performing the change of variables $u=f_k(v)$ and setting
$g(x,s)=\lambda s-b(x)s^{p-1}$ if $s\geq0$, $x\in\Omega$ and $g(x,s)=0$ for $s<0$, $x\in\Omega$,
we obtain that problem \eqref{P} is equivalent to the following
semilinear elliptic equation
\begin{equation}\label{P'}
    \left\{ \begin{aligned}
         -\Delta v &= \lambda f_\kappa(v)f_\kappa'(v) - b(x) (f_\kappa(v))^pf_\kappa'(v) &\mbox{in}&\quad\Omega,\\
         v&=0 &\mbox{on}&\quad\partial \Omega.
    \end{aligned} \right.
\end{equation}
Furthermore, we can see that $v$ is a classical positive solution of \eqref{P'} if and only if $u=f_k(v)$ is a positive
solution of \eqref{P} (see \cite{Jeanjean2014,Liuwang}). Thus,  we will analyze the
auxiliary problem \eqref{P'}.

In order to study problem \eqref{P'}, we need to establish some properties of the change of variable $f_\kappa(t)$. Firstly, we recall some useful properties of $f_\kappa(t)$ (see for instance \cite{watanabe2012,Jeanjean2014}).
\begin{lemma}\label{lema:f}
    Let $\kappa>0$ and $t \geq 0$. Then
    \begin{description}
      \item[$(i)$] $0 \leq f_\kappa(t)\leq t$;
      \item[$(ii)$] $0  \leq  f_\kappa'(t) \leq 1$;
      \item[$(iii)$] $f_\kappa(t)f_\kappa'(t) \leq 1/\sqrt{2\kappa}$;
      \item[$(iv)$] $f_\kappa''(t) = - 2 \kappa f_\kappa(t) (f_\kappa'(t))^4 = [(f_\kappa'(t))^4 - (f_\kappa'(t))^2]/f_\kappa(t)$;
      \item[$(v)$] $\frac{1}{2} f_\kappa(t) \leq t f_\kappa'(t) \leq f_\kappa(t)$;
      \item[$(vi)$] $\displaystyle\lim_{t \rightarrow 0^+}f_\kappa(t)/t=1$;
      \item[$(vii)$] The map $t \in (0,\infty) \mapsto f_\kappa(t)/t^{1/2}$ is nondecreasing.
    \end{description}
\end{lemma}
As a consequence of Lemma~\ref{lema:f}, we also have the following properties.

\begin{lemma}\label{Prop:f} Assume that $\kappa>0$ and $p>1$. Then
\begin{description}
    \item[$(i)$] The map $t \in (0,+\infty) \mapsto f_\kappa(t) f_\kappa'(t)/t$ is of
    class $\mathcal{C}^1$, decreasing and it verifies
    \begin{equation}\label{cotaff'}
        f_\kappa(t) f_\kappa'(t) \leq t, \quad \forall\ t \geq 0,
    \end{equation}
    \begin{equation}\label{lim0}
        \lim_{t \rightarrow 0^+} \frac{f_\kappa(t) f_\kappa'(t)}{t} = 1
    \end{equation}
    and
    \begin{equation}\label{liminfity}
            \lim_{t \rightarrow \infty} \frac{f_\kappa(t) f_\kappa'(t)}{t} = 0;
    \end{equation}
    \item[$(ii)$] For $p\geq3$, the map $t \in (0,\infty) \mapsto f_\kappa^p(t)f_\kappa'(t)/t$ is of class $\mathcal{C}^1$, increasing and it verifies
    \begin{equation}\label{limp}
         \lim_{t \rightarrow 0^+} \frac{f_\kappa^p(t) f_\kappa'(t)}{t}  = 0.
    \end{equation}
    \end{description}
\end{lemma}
\begin{proof}
To prove that $t \in (0,+\infty) \mapsto f_\kappa(t) f_\kappa'(t)/t$ is decreasing note that, for all $t>0$,
$$
\begin{aligned}
\left(\frac{f_\kappa(t) f_\kappa'(t)}{t} \right)' &= \frac{[(f_\kappa'(t))^2 + f_\kappa(t)f_\kappa''(t)]t
- f_\kappa(t)f_\kappa'(t)}{t^2}\\
&=\frac{[(f_\kappa'(t))^2 - 2(f_\kappa(t))^2(f_\kappa'(t))^5]t - f_\kappa(t)f_\kappa'(t)}{t^2}<0,
\end{aligned}
$$
if and only if,
$
tf_\kappa'(t) < 2t(f_\kappa(t))^2(f_\kappa'(t))^4 + f_\kappa(t),
$
which is true, thanks to Lemma~\ref{lema:f} $(i)$, $(ii)$ and $(v)$. The inequality \eqref{cotaff'} is a direct consequence of
Lemma~\ref{lema:f} $(i)$  and $(ii)$. The limit \eqref{lim0} is obtained by combining Lemma~\ref{lema:f} $(vi)$ and using that
$$
\lim_{t \rightarrow 0^+} f_\kappa'(t) =  \lim_{t \rightarrow 0^+} \frac{1}{(1+2\kappa f_\kappa^2(t))^{1/2}} = 1.
$$
The limit \eqref{liminfity} follows from Lemma~\ref{lema:f} $(iii)$.

Now, suppose that $p\geq3$. To prove that the map $t \in [0,\infty) \mapsto  f_\kappa^p(t)f_\kappa'(t)/t$ is increasing,
we observe that, by Lemma~\ref{lema:f} $(iv)$, for all $t>0$ we have
$$
\begin{aligned}
\left(\frac{f_\kappa^p(t) f_\kappa'(t)}{t} \right)' &= \frac{[p(f_\kappa(t))^{p-1}(f_\kappa'(t))^2
+ f_\kappa^p(t)f_\kappa''(t)]t - f_\kappa^p(t)f_\kappa'(t)}{t^2}\\
&=\frac{[p(f_\kappa(t))^{p-1}(f_\kappa'(t))^2 +(f_\kappa(t))^{p-1}((f_\kappa'(t))^4 -
 (f_\kappa'(t))^2)]t - f_\kappa^p(t)f_\kappa'(t)}{t^2}>0,
\end{aligned}
$$
if and only if, $[p(f_\kappa'(t))^2 +(f_\kappa'(t))^4 - (f_\kappa'(t))^2]t - f_\kappa(t)f_\kappa'(t)>0$, that is,
\begin{equation}\label{753}
         t(f_\kappa'(t))^4 +(p-1)t(f_\kappa'(t))^2>f_\kappa(t)f_\kappa'(t).
\end{equation}
On the other hand, since $p\geq3$, it follows from Lemma \ref{lema:f} $(v)$ that
$$
tf_\kappa'(t) \geq \frac{f_\kappa(t)}{2} \geq \frac{f_\kappa(t)}{p-1}, \quad \forall\ t \geq 0.
$$
Now, using Lemma~\ref{lema:f} $(ii)$ in combination with the fact that $t(f_\kappa'(t))^4 >0$ for all $t>0$,
we conclude that \eqref{753} is true. Finally, \eqref{limp} is an easy consequence of  \eqref{lim0} and $\lim_{t \rightarrow 0^+} f^{p-1}(t)=f^{p-1}(0)=0$.
\end{proof}

With respect to the map $\kappa \in (0,\infty) \mapsto f_\kappa(t)$ (for each $t \geq 0$), we have the following lemma.
\begin{lemma}
For each $t> 0$, the function  $\kappa \in (0,\infty) \mapsto f_\kappa(t)$ is continuous and decreasing.
\end{lemma}
\begin{proof}
The continuity of the map $\kappa \in (0,\infty) \mapsto f_\kappa(t)$ follows from the standard theory of ordinary differential equations. To prove that it is decreasing, we argue as follows.  Let $\kappa_1,\kappa_2$ be constants such that $0<\kappa_1<\kappa_2$. We need prove that $f_{\kappa_2}(t) < f_{\kappa_1}(t)$ for all $t> 0$.
Since, for each $t>0$, the function $\kappa \mapsto f_\kappa^{-1}(t) = \int_0^t (1+2\kappa s^2)^{1/2}ds$ is increasing, it suffices to prove that
\begin{equation}\label{tia}
f_{\kappa_2}^{-1}(f_{\kappa_2}(t)) <f_{\kappa_2}^{-1}(f_{\kappa_1}(t)),
\end{equation}
which is equivalent to $t < \int_0^{f_{\kappa_1}(t)}(1+2\kappa_2 s^2)^{1/2}ds$. To this, consider the function defined by
$$
h(t)= \int_0^{f_{\kappa_1}(t)}(1+2\kappa_2 s^2)^{1/2}ds - t, \quad t \geq 0
$$
and notice that $h(0)=0$. We claim that $h'(t)>0$ for all $t>0$ which implies that $h(t)>0$ and hence \eqref{tia} holds. Indeed, observe that
$h'(t) = (1+2\kappa_2 f_{\kappa_1}^2(t))^{1/2} f_{\kappa_1}'(t) - 1>0$
if and only if
$$
\frac{1}{(1+2\kappa_1 f_{\kappa_1}^2(t))^{1/2}}=f_{\kappa_1}'(t) > \frac{1}{(1+2\kappa_2 f_{\kappa_1}^2(t))^{1/2}},
$$
which holds if $\kappa_1<\kappa_2$ and this completes the proof.
\end{proof}
We finish this section by deriving an \textit{a priori} estimate for positive solutions of \eqref{P'} in the particular
case $b(x) \equiv b>0$. This estimate will be useful to prove a nonexistence result in the next section.

\begin{lemma}\label{1a-cota}
Let $v \in \mathcal{C}^{2}(\overline{\Omega})$ be a positive solution of \eqref{P'} with $b(x) \equiv b> 0$ constant. Then
\begin{equation}\label{a priori}
bf_\kappa^{p-1}(v(x)) \leq \lambda, \quad \forall\ x \in \Omega.
\end{equation}
\end{lemma}
\begin{proof}
Let $v$ be a classical positive solution of \eqref{P'}. Since the maximum value of $v$ in $\overline{\Omega}$  is attained in $\Omega$, let $x_0\in \Omega$ be such that $v(x_0)= \max_{x \in \overline{\Omega}}v(x)$. Thus,
$$
0 \leq -(\Delta v)(x_0) = \lambda f_\kappa(v(x_0))f_\kappa'(v(x_0)) - b f_\kappa^pv(x_0))f_\kappa'(v(x_0))
$$
and as $f_\kappa(v(x_0))f_\kappa'(v(x_0)) > 0$, the previous inequality is equivalent to
$
bf_\kappa^{p-1} (v(x_0)) \leq \lambda.
$
Using that $f_\kappa(t)$ is increasing for $t>0$, we obtain
$bf_\kappa^{p-1}(v(x)) \leq bf_\kappa^{p-1}(v(x_0)) \leq \lambda$ for all $x \in \Omega$, and this completes the proof.
\end{proof}

\section{Existence, nonexistence and uniqueness of positive solutions}\label{sec3}

In this section, we will study the existence, nonexistence and uniqueness of positive solutions of \eqref{P'}.
We begin by establishing a necessary condition for existence of positive solution for \eqref{P'} (and hence for \eqref{P}).

\begin{lemma}[Nonexistence]\label{nexistencia}
If $(b_0)$ holds then problem \eqref{P'} does not have positive solutions for $\lambda\leq\lambda_1$. In particular, if $b(x)\equiv 0$ then problem \eqref{P'} does not have positive solutions for $\lambda\leq\lambda_1$.
\end{lemma}
\begin{proof}
Suppose that $v>0$ is a solution of \eqref{P'} with $\lambda \leq \lambda_1$. Then, it satisfies
\begin{equation*}
    \left\{ \begin{aligned}
         -\Delta v + \widetilde{b}(x)v &= 0  &\mbox{in}&\quad\Omega,\\
         v&=0  &\mbox{on}&\quad\partial \Omega,
    \end{aligned} \right.
\end{equation*}
where
$$
\widetilde{b}(x):=b(x) \frac{f_\kappa^p(v(x))f_\kappa'(v(x))}{v(x)} - \lambda \frac{f_\kappa(v(x))f_\kappa'(v(x))}{v(x)}.
$$
Thus, we can infer that
$\lambda_1[ -\Delta  + \widetilde{b}(x) ]=0$.
Since $b(x) f_\kappa^p(v(x))f_\kappa'(v(x))/v(x)\geq0$, using the monotonicity properties of the principal eigenvalue
combined with \eqref{cotaff'}, we conclude that $0>\lambda_1\left[ -\Delta   - \lambda \right] = \lambda_1 - \lambda$,
which is a contradiction and this ends the proof.
\end{proof}

The next result shows an uniqueness result of positive solutions of \eqref{P'}.

\begin{proposition}[Uniqueness]\label{unicidade}
Suppose $p\geq 3$ or $b(x) \equiv b>0$. Then the problem \eqref{P'} admits at most a positive solution.
\end{proposition}
\begin{proof}
First we will consider the case $p \geq 3$. By the classical Brezis-Oswald result (see \cite{BO}), it is sufficient to prove that the function
$$
h(x,t):=\lambda\frac{f_\kappa(t)f_\kappa'(t)}{t} - b(x)\frac{f_\kappa^p(t)f_\kappa'(t)}{t}
$$
is decreasing in $t>0$, for each $x\in \Omega$. Thus, the monotonicity follows by Lemma \ref{Prop:f}.

Now, assume that $b(x) \equiv b>0$ is constant. We will argue by contradiction.
Suppose that $v_1>0$ and $v_2>0$  are solutions of (\ref{P'}) with $v_1 \neq v_2$. Denoting, by simplicity, $g_i = f_\kappa(v_i)$ and $g_i' = f_\kappa'(v_i)$ ($i=1,2$), we have $-\Delta(v_1-v_2) = \lambda(g_1g_1' - g_2g_2') - b (g_1^pg_1' - g_2^pg_2')$.
Defining $W:\Omega\rightarrow \mathbb{R}$ by
\begin{equation*}
    W(x)=\left\{\begin{array}{cc}
         \displaystyle\frac{-\lambda[g_1(x)g_1'(x) - g_2(x)g_2'(x)] + b [g_1^p(x)g_1'(x) - g_2^p(x)g_2'(x)]}{v_1(x)-v_2(x)}, & \mbox{if}\ \ v_1(x)\neq v_2(x),  \\
         0&  \mbox{if}\ \ v_1(x)= v_2(x),
    \end{array}\right.
\end{equation*}
we have $-\Delta(v_1-v_2) + W(x) (v_1-v_2)  = 0$ in $\Omega$. Consequently, $\lambda_j[-\Delta + W(x)]=0$,
where $\lambda_j[-\Delta + W(x)]$, $j\geq 1$, stands for an eigenvalue of $-\Delta +W(x)$ in $\Omega$ under homogeneous Dirichlet boundary conditions. By the dominance of the principal eigenvalue, we get
\begin{equation}\label{Wx}
    0= \lambda_j[-\Delta + W(x)] \geq\lambda_1[-\Delta + W(x)].
\end{equation}
On the other hand, since $v_1$ is a positive solution of (\ref{P'}), we have
\begin{equation}\label{v1}
\lambda_1\left[-\Delta - \lambda \frac{g_1g_1'}{v_1} + b \frac{g_1^pg_1'}{v_1}\right]=0.
\end{equation}
We claim that
\begin{equation}\label{afirmacao}
    - \lambda \frac{g_1g_1'}{v_1} + b \frac{g_1^pg_1'}{v_1}\leq W \quad \mbox{in}\quad\Omega,
\end{equation}
with strict inequality in an open subset of $\Omega$. If (\ref{afirmacao}) holds, then the proof is completed because we can combine (\ref{v1})-(\ref{afirmacao}) and the monotonicity properties of the principal eigenvalue to obtain
$$
0= \lambda_1\left[-\Delta - \lambda \frac{g_1g_1'}{v_1} + b \frac{g_1^pg_1'}{v_1}\right] < \lambda_1[-\Delta + W(x)],
$$
which contradicts (\ref{Wx}). Now, let us prove (\ref{afirmacao}). If $v_1(x)=v_2(x)$ then $W(x)=0$ and (\ref{afirmacao}) is equivalent to
$$
- \lambda \frac{g_1(x)g_1'(x)}{v_1(x)} + b \frac{g_1^p(x)g_1'(x)}{v_1(x)} \leq 0
$$
that is, $b g_1^{p-1}\leq \lambda$ in $\Omega$, which occurs thanks to Lemma~\ref{1a-cota}. If $v_1>v_2$, then $v_1-v_2>0$ and (\ref{afirmacao}) is equivalent to
\begin{equation*}
    -\lambda g_1g_1' + b g_1^pg_1' + \lambda \frac{g_1g_1'}{v_1}v_2 - b \frac{g_1^pg_1'}{v_1}v_2  \leq -\lambda g_1g_1' +\lambda g_2g_2' + b g_1^pg_1' -  b g_2^pg_2'\quad \mbox{in}\quad\Omega,
\end{equation*}
that is,
\begin{equation}\label{afeq}
[\lambda - b g_1^{p-1}]\frac{g_1g_1'}{v_1} \leq [\lambda - b g_2^{p-1}]\frac{g_2g_2'}{v_2}\quad \mbox{in}\quad\Omega.
\end{equation}
Since the map $t\in [0,\infty) \mapsto f_\kappa(t)f_\kappa'(t)/t$ is decreasing, we have
\begin{equation}\label{af1}
0 \leq     \frac{g_1g_1'}{v_1} < \frac{g_2g_2'}{v_2}.
\end{equation}
Once that $t\in [0,\infty) \mapsto f_\kappa(t)f_\kappa'(t)$ is increasing and by using Lemma \ref{1a-cota}, we can infer that
\begin{equation}\label{af2}
    0\leq\lambda - b g_1^{p-1} \leq \lambda - b g_2^{p-1}\quad \mbox{in}\quad\Omega.
\end{equation}
Thus, (\ref{af1}) and (\ref{af2}) imply that (\ref{afeq}) is true, showing that (\ref{afirmacao}) holds for $v_1>v_2$.
The case $v_1<v_2$ is analogous and this ends the proof.

\end{proof}


Now, we will show that $\lambda_1$ is the unique  bifurcation point of positive solutions of \eqref{P'} from the trivial solution.
For this, let $e_1$ be the unique positive solution of
$$
    \left\{\begin{aligned}
        -\Delta v &= 1  &\mbox{in}&\quad\Omega,  \\
         v&=0& \mbox{on}&\quad \partial\Omega,
    \end{aligned} \right.
$$
and let $E$ be the space consisting of all $u \in \mathcal{C}(\overline{\Omega})$ for which there exists $\gamma = \gamma_u>0$ such that
$$-\gamma e_1(x) \leq u(x) \leq \gamma e_1(x) \quad \forall\ x \in \Omega, $$
endowed with the norm $\|u\|_E:= \mbox{inf}\{\gamma>0;~ -\gamma e_1(x) \leq u(x) \leq \gamma e_1(x), ~\forall x \in \Omega\}$
and the natural point-wise order. It is not difficult to verify that $E$ is an ordered Banach space whose positive cone, say $P$, is normal and has nonempty interior. Thus, consider the map $\mathfrak{F}: \R \times E \longrightarrow E$ defined by
\begin{eqnarray*}
\mathfrak{F}(\lambda,v) = v- (-\Delta)^{-1}[\lambda f_\kappa(v)f_\kappa'(v) - b(x) f_\kappa^p(v)f_\kappa'(v)],
\end{eqnarray*}
where $(-\Delta)^{-1}$ is the inverse of the Laplacian operator under homogeneous Dirichlet boundary conditions. We can see that the application $\mathfrak{F}$ is of $\mathcal{C}^1$ class
 and \eqref{P'}  can be written in the form
\begin{equation}
\label{cero}
\mathfrak{F}(\lambda,v)=0.
\end{equation}
Moreover, by the Strong Maximum Principle, any nonnegative solution of \eqref{cero} is in fact strictly positive.

\begin{proposition}\label{bifur}
The number $\lambda_1$ is a bifurcation point of \eqref{P'} from the trivial solution to a continuum of positive solutions of  \eqref{P'}. Moreover, it is the unique bifurcation point of positive solutions from $(\lambda,0)$.  If  $\Sigma_0 \subset \mathcal{S}$ denotes the component of positive solutions of \eqref{P'}  emanating from $(\lambda,0)$, then $\Sigma_0$ is unbounded in $\R\times E$.
\end{proposition}
\begin{proof}
Observe that \eqref{cero} can be written as $
{\cal L}(\lambda)v+{\cal N}(\lambda,v)=0$
where ${\cal L}(\lambda)=I_E-\lambda (-\Delta)^{-1}$ and
$$
{\cal N}(\lambda,v)=- (-\Delta)^{-1}[\lambda (f_\kappa(v)f_\kappa'(v)-v) - b(x) f_\kappa^p(v)f_\kappa'(v)].
$$
Moreover, thanks to \eqref{lim0} and \eqref{limp}, we have
$$
\lim_{t \rightarrow 0^+}\frac{\lambda (f_\kappa(t)f_\kappa'(t)-t) - b(x)f_\kappa^p(t)f_\kappa'(t)}{t} = 0,
$$
and thus ${\cal N}(\lambda,v)=o(\|v\|_E)$ as $\|v\|_E\to 0$. Therefore, we can apply the unilateral bifurcation theorem
for positive operators, see \cite[Theorem 6.5.5]{Bifbook}, to conclude the result.
\end{proof}

Next, we are ready to complete the proof of Theorem~\ref{T1}. Actually, it will be a consequence of the following result.

\begin{theorem}
Let $p>1$, $\kappa>0$ and assume $(b_0)$. Then problem \eqref{P'} possesses  a positive solution if and only if $\lambda> \lambda_1$. Moreover, if $p \geq 3$ or $b(x)\equiv b>0$ is a constant,  it is unique if it exists and it will be denoted by $\Theta_{\lambda,\kappa}$. In addition, the map $\lambda \in (\lambda_1,+\infty) \mapsto \Theta_{\lambda,\kappa}  \in \mathcal{C}_0^1(\overline{\Omega})
$
is  increasing, in the sense that $\Theta_{\lambda,\kappa} > \Theta_{\mu,\kappa}$, if $\lambda> \mu > \lambda_1.$ Furthermore,
$\lim_{\lambda \downarrow \lambda_1}\|\Theta_{\lambda,\kappa}\|_\infty  =0$
and for any compact $K \subset \overline{\Omega}_{b,0} \setminus \partial\Omega$,
\begin{equation*}
    \lim_{\lambda \rightarrow +\infty} \Theta_{\lambda,\kappa} = \infty \quad \mbox{uniformly in }K.
\end{equation*}
\end{theorem}

\begin{proof}
By Proposition~\ref{bifur}, $\lambda_1$ is a bifurcation point of \eqref{P'} from the trivial solution and it is the only one for positive solutions. Moreover, there exists an unbounded continuum $\Sigma_0$ of positive solutions emanating from $(\lambda_1,0)$. In order to prove the existence of a positive solution for every $\lambda >\lambda_1$, it suffices to show that, for every $\lambda_*>\lambda_1$, there exists a constant $C=C(\lambda_*)>0$ such that
\begin{equation}\label{vcota}
\|v\|_\infty \leq C, \quad \forall\ (\lambda, v) \in \Sigma_0\ \ \mathrm{and}\ \  \lambda \leq \lambda_* .
\end{equation}
Indeed, by the global nature of $\Sigma_0$, this estimate implies that $\mbox{Proj}_{\R}\Sigma_0 = (\lambda_1, \infty)$,
where $\mbox{Proj}_{\R} \Sigma_0$ is the projection of $\Sigma_0$ into $\R$. To prove \eqref{vcota}, we will build a family $\overline{W}(\lambda)$ of super solutions of \eqref{P'} and we will apply Theorem 2.2 of \cite{LGsub}. Thus, we consider the continuous map $\overline{W}:[\lambda_1, \lambda_*] \rightarrow \mathcal{C}^2_0(\overline{\Omega})$ defined by $\overline{W}(\lambda) = K(\lambda)e$, where $K(\lambda)$ is a positive constant to be chosen later and $e$ is the unique positive solution of
\begin{eqnarray}\label{e}
\left\{\begin{aligned}
 -\Delta v& = 1 & \mbox{in}&
 \quad \widehat{\Omega}, \\
 v&=0& \mbox{on}&\quad \partial \widehat{\Omega},
\end{aligned}\right.
\end{eqnarray}
for some regular domain $\Omega \subset\subset \widehat{\Omega}$. Then, $\overline{W}(\lambda) = K(\lambda)e$ is a
super solution of \eqref{P'} if
\begin{eqnarray*}
1\geq \lambda \frac{f_\kappa(e)f_\kappa'(e)}{Ke} e -b(x) \frac{f_\kappa^p(e)f_\kappa'(e)}{Ke} e \quad \mbox{in}\quad\Omega.
\end{eqnarray*}
According to Proposition~\ref{Prop:f},
$\lim_{t \rightarrow \infty}f_\kappa(t)f_\kappa'(t)/t = 0$.
Consequently, for $K =K(\lambda) >0$ large enough, $\overline{W}(\lambda)=K(\lambda)e$ is a super solution
(but not a solution) of \eqref{P'}, for every $\lambda \in [\lambda_1, \lambda_*]$ and $W(\lambda_1)
=K(\lambda_1)e> 0$ in $\Omega$. Thus, by Theorem 2.2 of \cite{LGsub}, it follows \eqref{vcota}.

To prove that $\Theta_{\lambda,\kappa} > \Theta_{\mu,\kappa}$ if $\lambda > \mu > \lambda_1$,
just note that $\Theta_{\mu,\kappa}$ is a (strict) subsolution of \eqref{P'} if $\mu \in (\lambda_1, \lambda)$.
By the uniqueness of positive solution of \eqref{P'} we conclude the result.

The convergence \eqref{liml1} is an immediate consequence of Proposition~\ref{bifur}.

Now, in order to prove \eqref{vinfinito}, let $\varphi_{b,0}>0$ be the eigenfunction associated to $\lambda_{b,0}$
such that $\|\varphi_{b,0}\|_\infty=1$ and consider
$$
\Psi=\left\{
\begin{array}{ll}
\varphi_{b,0} & \mbox{in\quad $\Omega_{b,0}$,} \\
0  & \mbox{in\quad $\Omega\setminus \overline{\Omega}_{b,0}$.}
\end{array}\right.
$$
It is clear that $\Psi\in H_0^1(\Omega)$.
We will show that for $\lambda > \lambda_{b,0}$, $\varepsilon(\lambda)\Psi$ is a subsolution of \eqref{P'}
(in the sense of \cite{bl}) for a constant $\varepsilon(\lambda)>0$ to be chosen. Indeed, since $b\equiv 0$
in $\Omega_{b,0}$ and $\Psi=0$ in $\Omega\setminus \overline{\Omega}_{b,0}$, it suffices to verify that
$$
\lambda_{b,0}\varepsilon \varphi_{b,0}=-\Delta(\varepsilon \varphi_{b,0}) \leq
\lambda f_\kappa(\varepsilon \varphi_{b,0})f_\kappa'(\varepsilon \varphi_{b,0})
\quad \mbox{in}\quad\Omega_{b,0},
$$
that is,
$$
\frac{\lambda_{b,0}}{\lambda} \leq \frac{f_\kappa(\varepsilon \varphi_{b,0})
f_\kappa'(\varepsilon \varphi_{b,0})}{\varepsilon \varphi_{b,0}} \quad \mbox{in}\quad\Omega_{b,0}.
$$
According to Lemma~\ref{Prop:f}, the map $t \in [0,\infty) \mapsto h_\kappa(t):= f_\kappa(t)f_\kappa'(t)/t$ is
decreasing and, hence, is invertible. Then, the above inequality is equivalent to $
h_\kappa^{-1}(\lambda_{b,0}/\lambda) \geq \varepsilon \varphi_{b,0}.
$
Once that $\|\varphi_{b,0}\|_\infty=1$, choosing $\varepsilon(\lambda):= h^{-1}(\lambda_{b,0}/\lambda),$ we obtain that $\varepsilon (\lambda) \varphi_{b,0}$ is a subsolution of \eqref{P'}. Moreover,
it follows from \eqref{lim0} that $\lim_{t \rightarrow 0} h^{-1}(t) = +\infty$ and therefore
\begin{equation}\label{epsilon}
    \lim_{\lambda \rightarrow \infty} \varepsilon(\lambda) = \lim_{\lambda \rightarrow \infty}
    h^{-1}\left(\frac{\lambda_{b,0}}{\lambda}\right)=+\infty.
\end{equation}
Now, the previous arguments establish that $K(\lambda)e$ is a super solution of \eqref{P'} for
all $K$ large enough. Thus, since $\min_{x \in \overline{\Omega}}e(x)>0$, we can choose $K$ such that
$\varepsilon (\lambda) \varphi_{b,0} \leq K(\lambda)e$. Therefore, by method of sub and supersolution
and the uniqueness of positive solution of \eqref{P'}, we can infer that
$\varepsilon (\lambda) \varphi_{b,0} \leq  \Theta_{\lambda,\kappa}$.
Consequently, by \eqref{epsilon} we obtain \eqref{vinfinito} and this complete the proof.
\end{proof}

We observe that as a direct consequence of this result, the proof of Theorem~\ref{T1} follows by setting $\Psi_{\lambda,\kappa}:=f_\kappa(\Theta_{\lambda,\kappa})$.

\section{Stability Result}\label{sec4}
In this section we will provide the stability of the positive solutions of \eqref{P'}  with the additional assumption
that $p\geq3$. We recall that the stability of a positive solution $(\lambda_0,u_0)$ of \eqref{P'} as a steady state of an
associated parabolic equation is given by the spectrum of the linearized operator of \eqref{P'}, which is
$$
\mathcal{L}(\lambda_0,u_0):=-\Delta  -\lambda_0 [f_\kappa(u_0)f_\kappa'(u_0)]'+b(x) [f_\kappa^p(u_0)f_\kappa'(u_0)]',
$$
subject to homogeneous Dirichlet boundary conditions on $\partial\Omega$, where $'=d/dt$. Thus, $(\lambda_0,u_0)$ is
said to be  linearly asymptotically stable if $\lambda_1[\mathcal{L}(\lambda_0,u_0)]>0$.

With these considerations, we have the following result:

\begin{proposition}\label{estabilidade}
Suppose $p\geq3$ or $b(x)\equiv b>0$. Then, for each $\lambda>\lambda_1$ and $\kappa>0$, the unique positive solution $(\lambda,\Theta_{\lambda,\kappa})$ of \eqref{P'} is linearly asymptotically stable, that is,
$$
\lambda_1[\mathcal{L}(\lambda,\Theta_{\lambda,\kappa})]>0.
$$
\end{proposition}
\begin{proof}
Denoting by simplicity $f=f_\kappa(\Theta_{\lambda,\kappa})$ and using Lemma~\ref{lema:f} (iv), we have
$$
[ff']' = (f')^2 +ff'' =  (f')^2 -2\kappa f(f')^4 = (f')^4,
$$
and
$$[f^pf']'=[f^{p-1}(ff')]'  = (p-1)f^{p-2}f'(ff') + f^{p-1} (f')^4 = f^{p-1}[(p-1)(f')^2 + (f')^4].
$$
Therefore,
$$
\mathcal{L}(\lambda,\Theta_{\lambda,\kappa}) = -\Delta  -\lambda [(f')^2 -2\kappa f(f')^4]  + b(x) f^{p-1}[(p-1)(f')^2 + (f')^4].
$$
By the characterization of the Maximum Principle, see for instance \cite[Theorem~2.1]{LG-MM} or \cite{SMPbook}, to prove $\lambda_1[\mathcal{L}(\lambda,\Theta_{\lambda,\kappa})]>0$ it is sufficient to show that there exists a positive strict super solution of $\mathcal{L}(\lambda,\Theta_{\lambda,\kappa})$. Let us prove  that $\Theta_{\lambda,\kappa}$ is a strict super solution of $\mathcal{L}(\lambda,\Theta_{\lambda,\kappa})$. Indeed, since $\Theta_{\lambda,\kappa}$ is a positive solution of \eqref{P'}, we have
$
-\Delta \Theta_{\lambda,\kappa} = \lambda ff' - b(x)f^pf'
$
and therefore
\begin{equation}\label{Ltheta}
\begin{aligned}
\mathcal{L}(\lambda,\Theta_{\lambda,\kappa})\Theta_{\lambda,\kappa}
&= \lambda ff' - b(x) f^pf' -\lambda (f')^2\Theta_{\lambda,\kappa}\\
&\ \ \ +
2\lambda\kappa f(f')^4 \Theta_{\lambda,\kappa}  + b(x) f^{p-1}[(p-1)(f')^2 + (f')^4] \Theta_{\lambda,\kappa}.
\end{aligned}
\end{equation}
Since $p\geq 3$, it follows from Lemma~\ref{lema:f} (v) that
\begin{equation}\label{1111}
\begin{array}{c}
         f-f'\Theta_{\lambda,\kappa} = f(\Theta_{\lambda,\kappa})-f'(\Theta_{\lambda,\kappa})\Theta_
         {[\lambda,\kappa]} >0\quad \mathrm{and}  \\
     (p-1)f'\Theta_{\lambda,\kappa} - f= (p-1)f'(\Theta_{\lambda,\kappa})\Theta_{\lambda,\kappa} -
     f(\Theta_{\lambda,\kappa}) >0.
\end{array}
\end{equation}
Moreover,  since $b(x)f^{p-1}(f')^3\Theta_{\lambda,\kappa} \geq0$ and $ 2\lambda\kappa
f(f')^4 \Theta_{\lambda,\kappa} \geq 0$, we can infer from \eqref{Ltheta} and \eqref{1111} that
$$
\mathcal{L}(\lambda,\Theta_{\lambda,\kappa})\Theta_{\lambda,\kappa}>0,
$$
which establishes that $\Theta_{\lambda,\kappa}>0$ is a strict positive super solution of
$\mathcal{L}(\lambda,\Theta_{\lambda,\kappa})$. By characterization of the Maximum Principle,
$\lambda_1[\mathcal{L}(\lambda,\Theta_{\lambda,\kappa})]>0$ and this completes the proof.

\end{proof}
As a direct consequence we obtain:
\begin{corollary}
Supposing $p\geq3$, we have:
\begin{description}
    \item[$(i)$] For each $\lambda > \lambda_1$,  $(\lambda,\Theta_{\lambda,\kappa})$ is a nondegenerate positive solution
    of \eqref{P'}.
    \item[$(ii)$] The map $\lambda \in (\lambda_1,+\infty) \mapsto \Theta_{\lambda,\kappa} \in \mathcal{C}_0^1(\Omega)$
    is of class $\mathcal{C}^\infty$.
\end{description}
\end{corollary}
\begin{proof}
The proof of $(i)$ is standard and, once that $t \in [0,+\infty) \mapsto f_k(t)$ is of class $\mathcal{C}^\infty$, $(ii)$
follows from implicit function theorem applied to the operator
$$
\mathfrak{F}(\lambda,u):= u - (-\Delta)^{-1}[\lambda f_k(u)f_k'(u) - b(x)f_k^p(u)f_k'(u)].
$$
\end{proof}
ection{A priori bounds in $\Omega_+$}\label{sec5}
This section is devoted to obtain an \textit{a priori} estimate for positive solutions of \eqref{P'} uniform in $\kappa>0$,
$\kappa \simeq 0$ in any compact subset of $\Omega_+$. It is a crucial step to prove Theorem~\ref{T2} $(c)$. As we will
see below, to obtain these estimates we will assume $p>3$. To this aim,  we need to study the following  auxiliary problem
\begin{equation}\label{Pinfinito}
    \left\{ \begin{aligned}
        -\Delta v &= \lambda v - b_0 g(v) &\mbox{in}&\quad B_r, \\
         v&=\infty&\mbox{on}&\quad\partial B_r,
    \end{aligned}
    \right.
\end{equation}
where $b_0>0$ is a constant, $B_r := B_r(x_0) = \{x \in \R^N;~ |x-x_0|< r\}$ is an open ball in $\R^N$ centered in $x_0 \in \R^N$ and
\begin{equation}\label{g}
g(t):=\frac{f_1^{p+1}(t)}{t},\quad\forall\ t>0.
\end{equation}

First we will prove some important properties of $g$.

\begin{lemma}\label{lema:g}
The map $g:(0,\infty) \rightarrow (0,\infty)$ defined in \eqref{g}
is increasing and it satisfies $g(0):=\lim_{t \rightarrow 0^+}g(t)=0$. Moreover,  there exists a constant $C>0$ such that
    \begin{equation}\label{gcota}
    g(t) \geq C t^{(p-1)/2}, \quad \forall\ t \geq 1.
    \end{equation}
    Furthermore,
    \begin{equation}\label{cotafpf'}
         f_\kappa^p(t)f_\kappa'(t) \leq g(t),\quad \forall\ t >0\quad \textrm{and}\quad 0<\kappa<1.
    \end{equation}
\end{lemma}
\begin{proof}
In order to prove that $g$ is increasing, note that, by Lemma~\ref{lema:f} $(iii)$, we have
$$
tf_1(t) \geq \frac{f_1(t)}{2} > \frac{f_1(t)}{p+1}, \quad \forall\ t>0,
$$
since $p>1$. Thus,
$$
g'(t)=\left(\frac{f_1^{p+1}(t)}{t}\right)'= \frac{(p+1)f_1^p(t)t - f_1^{p+1}(t)}{t^2} > 0, \quad \forall\ t>0.
$$
To conclude the proof of inequality \eqref{gcota}, observe that for each $t>0$ one has
$$
\frac{g(t)}{t^{(p-1)/2}} = \left(\frac{f_1(t)}{t^{1/2}}\right)^{p+1}.
$$
By Lemma~\ref{lema:f} $(vii)$, $t \mapsto g(t)/t^{(p-1)/2}$ is nondecreasing and thus
$$
\frac{g(t)}{t^{(p-1)/2}} \geq g(1), \quad \forall\ t \geq 1.
$$
Choosing $C= g(1)$, we obtain \eqref{gcota}. Moreover, $\lim_{t \rightarrow 0^+} f_\kappa^p(t) (f_\kappa(t)/t)=g(0) =0$. Finally, combining the monotonicity of $\kappa \mapsto f_\kappa(\cdot)$ with Lemma~\ref{lema:f} $(v)$, we get
$$
f_\kappa^p(t)f_\kappa'(t) \leq \frac{f_\kappa^{p+1}(t)}{t} < \frac{f_1^{p+1}(t)}{t} = g(t), \quad \forall\ t > 0,
$$
which is the desired result.
\end{proof}

Now, we will establish an existence result for \eqref{Ploglargo}. We recall that there are many result about the existence,
uniqueness and blow-up rate of large solution of problems related to \eqref{Pinfinito}, see for instance,
\cite{CR,Yihong,JRJ} and references therein. The following lemma is a consequence of these works.

\begin{lemma}\label{Pbauxiliar}
\begin{description}
    \item[$(i)$] Let $\lambda,b_0,M$ be positive constants and consider the  following nonlinear boundary value problem
\begin{equation}\label{PM}
    \left\{ \begin{aligned}
        -\Delta v &= \lambda v - b_0 g(v) &\mbox{in}&\quad B_r, \\
         v&=M&\mbox{on}&\quad\partial B_r.
    \end{aligned}
    \right.
\end{equation}
Then \eqref{PM} has an unique positive solution denoted by $\Theta_{[\lambda,b_0,M,B_r]}$.
\item[$(ii)$] Suppose $p>3$. For each $x \in B_r$, the point-wise limit
$$
\Theta_{[\lambda,b_0,\infty,B_r]}(x):=\lim_{M\uparrow \infty} \Theta_{[\lambda,b_0,M,B_r]}(x)
$$
is well defined and it is a classical minimal positive solution of \eqref{Pinfinito}.
\end{description}
\end{lemma}
\begin{proof}
The existence of positive solution for \eqref{PM} can be easily obtained by the method of sub and supersolution and the uniqueness
follows from similar arguments used in Section~\ref{sec3}.

To prove (ii), we will apply Theorem 1.1 of \cite{CR}. Thus, it is sufficient to show that $g \in \mathcal{C}^1([0,\infty))$, $g\geq0$ and the map $t \in (0,+\infty) \mapsto g(t)/t$ is increasing and the Keller-Osserman condition, i.e.,
\begin{equation}\label{KO}
\int_1^\infty \frac{dt}{\sqrt{G(t)}} < \infty, \quad  \mbox{where }\quad G(t):=\int_0^tg(s)ds.
\end{equation}
Indeed, the regularity and positivity of $g$ is given by Lemma~\ref{lema:g}. To prove that $t \in (0,+\infty) \mapsto g(t)/t$
is increasing, note that
$$
\left(\frac{g(t)}{t}\right)'=\left(\frac{f_1^{p+1}(t)}{t^2}\right)' = \frac{(p+1)f_1^p(t)f_1'(t)t^2 - 2t f_1^{p+1}}{t^2}>0,
$$
if and only if $(p+1)t f_1'(t)>2f_1(t)$.
Since $p>3$, it follows from Lemma~\ref{lema:f} $(v)$ that
$$
(p+1)tf_1'(t)\geq\frac{(p+1)}{2}f_1(t)>2f_1(t),
$$
showing that $g(t)/t$ is increasing. Finally, observing  that \eqref{gcota} is a sufficient condition for
\eqref{KO} to occur, the proof is complete.
\end{proof}

Now, we are able to prove the main result of this section.

\begin{proposition}\label{cotakappa}
Suppose $p>3$. For each compact  $K \subset \Omega_+ = \{x \in \Omega;~b(x)>0\}$, there exists a constant $C=C(\lambda,K)>0$
such that $\|\Theta_{\lambda,\kappa}\|_{\mathcal{C}(K)} \leq C$ for all $\kappa \in (0,1)$.
Recall that $\Theta_{\lambda,\kappa}$ stands for the unique positive solution of \eqref{P'}.
\end{proposition}
\begin{proof}
Let $B_r:=B_r(x_0) \subset\subset \Omega_+$. In particular, $b_K:=\min_{x\in B_r}b(x)>0$.  By \eqref{cotaff'} and  \eqref{cotafpf'},
for all $0<\kappa<1$, $\lambda>\lambda_1$, $\Theta_{\lambda,\kappa}$ satisfies
$$
-\Delta \Theta_{\lambda,\kappa} = \lambda f_\kappa(\Theta_{\lambda,\kappa}) f_\kappa'(\Theta_{\lambda,\kappa}) -
b(x) f_\kappa^p(\Theta_{\lambda,\kappa}) f_\kappa'(\Theta_{\lambda,\kappa}) \leq \lambda \Theta_{\lambda,\kappa} -
b_K g(\Theta_{\lambda,\kappa}) \quad \mbox{in}\quad B_r.
$$
Thus, $\Theta_{\lambda,\kappa}$ is a subsolution of \eqref{PM} for all $M \geq \max_{B_r} \Theta_{\lambda,\kappa}$.
Since large constants are positive super solutions of \eqref{PM}, by the sub and supersolution method combined with the uniqueness
of positive solution of \eqref{PM}, we can infer that
$$
\Theta_{\lambda,\kappa} \leq \Theta_{[\lambda,M,b_K,B_r]} \quad  \mbox{in}\quad B_r, \quad \forall\ M \geq \max_{B_r}
\Theta_{\lambda,\kappa},~ 0<\kappa<1.
$$
Letting $M \rightarrow \infty$ in the above inequality, we get
$$
\Theta_{\lambda,\kappa} \leq \Theta_{[\lambda,\infty,b_K,B_r]} \quad  \mbox{in}\quad B_r; \quad  0<\kappa<1.
$$
In particular,
$$
\Theta_{\lambda,\kappa} \leq \Theta_{[\lambda,\infty,b_K,B_r]} \quad  \mbox{in}\quad B_{r/2}; \quad 0<\kappa<1.
$$
Consequently, setting $C:= \max_{B_{r/2}} \Theta_{[\lambda,\infty,b_K,B_r]}$, we obtain $\|\Theta_{\lambda,\kappa}\|_{\mathcal{C}(B_{r/2})} \leq C$. Observe that $C$ depends on $b_K:= \min_{x\in B_r}b(x)$, $B_r$ and $\lambda$. Finally, since $K$ can be covered by a finite union of
such balls, the proof is complete.
\end{proof}

\section{Proof of Theorem~\ref{T2}}\label{sec6}

In this section, we present the proof of Theorem~\ref{T2}. Some arguments used here are inspired in \cite{DLS2002}.
We point out that we will prove the results for the unique positive solution $\Theta_{\lambda,\kappa}$ of \eqref{P'}
and therefore we obtain a similar result for the unique positive solution
$\Psi_{\lambda,\kappa} = f_\kappa(\Theta_{\lambda,\kappa})$ of \eqref{P}.\\

\noindent\begin{demT2}
To prove $(a)$, we will apply the Implicit Function Theorem. Suppose $\lambda \in (\lambda_1,\lambda_{b,0})$. Note that,
for $\delta>0$ small enough, $\kappa \in [0,\delta) \mapsto f_\kappa(\cdot)$ is a continuous map and  $f_\kappa' = 1/(1+2 \kappa
f_\kappa^2)^{1/2}$, $\kappa \in [0,\delta) \mapsto f_\kappa'(\cdot)$ is also continuous. Therefore, we can consider a
continuous extension of $f_\kappa$ and $f_\kappa'$ at $(-\delta,\delta)$. Define
$\mathcal{F}:(-\delta,\delta)\times \mathcal{C}_0^1(\overline{\Omega}) \rightarrow \mathcal{C}_0^1(\overline{\Omega})$ by
$$
\mathcal{F}(\kappa,v)= v-(-\Delta)^{-1}[\lambda f_\kappa(v)f_\kappa'(v) - b (f_\kappa^p(v)f_\kappa'(v))].
$$
Thus, $\mathcal{F}(\kappa,v)$ is continuous in $\kappa$ and of class $\mathcal{C}^1$ in $v$. Moreover, the zeros of
$\mathcal{F}$ provide us the positive solution of \eqref{P'} if $\kappa>0$ and the positive solution of classical logistic
equation \eqref{Plog} if $\kappa=0$, since $f_0(t)=t$, $t \geq 0$.
Differentiating with respect to $v$ at $(0,\Theta_\lambda)$, we have
$$
D_v\mathcal{F}(0,\Theta_\lambda)v= v-(-\Delta)^{-1}[\lambda v - pb\Theta_\lambda^{p-1}v], \quad \forall\ v \in
\mathcal{C}_0^1(\overline{\Omega}).
$$
Since $\Theta_\lambda$ is a nondegenerate positive solution of \eqref{Plog}, the operator $\mathcal{F}(0,\Theta_\lambda)$
is an isomorphism. Thus, it follows from the Implicit Function Theorem that, for $\delta>0$ small, there exists a continuous map
$\kappa \in (-\delta,\delta) \mapsto v(\kappa) \in \mathcal{C}_0^1(\overline{\Omega})$
such that $v(0) = \Theta_\lambda$ and $\mathcal{F}(\kappa,v(\kappa))=0$ for each $\kappa \in (-\delta,\delta)$.
Observe that $v(\kappa)$ is a positive solution of \eqref{P'} for $\kappa>0$ and $\kappa \simeq 0$, since $\Theta_\lambda$
lies in the interior of the positive cone of $\mathcal{C}_0^1(\overline{\Omega})$. Consequently, by the uniqueness of
positive solution of \eqref{P'}, we obtain that $v(\kappa) = \Theta_{\lambda,\kappa}$. In particular, $\lim_{\kappa \downarrow 0} \Theta_{\lambda,\kappa} = \lim_{\kappa \downarrow 0} v(\kappa) = v(0) = \Theta_\lambda$, completing the proof of item $(a)$. Now, we will prove $(b)$. Suppose $\lambda \geq \lambda_{b,0}$. By the monotonicity of
$\lambda \mapsto \Theta_{\lambda,\kappa}$, for each $\varepsilon>0$ small enough, we have $\Theta_{\lambda_{b,0}-\varepsilon,\kappa} < \Theta_{\lambda,\kappa}$. Using the part $(a)$, we can infer that
$$
\Theta_{\lambda_{b,0}-\varepsilon} = \lim_{\kappa \downarrow 0}\Theta_{\lambda_{b,0}-\varepsilon,\kappa}
< \liminf_{\kappa\downarrow0}\Theta_{\lambda,\kappa}.
$$
Taking into account \eqref{logInfty}, we conclude that
$$
+\infty = \lim_{\varepsilon \rightarrow 0^+}\Theta_{\lambda_{b,0}-\varepsilon} \leq \liminf_{\kappa\downarrow0}
\Theta_{\lambda,\kappa} \quad \mbox{uniformly in compact subsets of } \overline{\Omega}_{b,0}\setminus \partial\Omega.
$$
Therefore,
$
\lim_{\kappa\downarrow0}\Theta_{\lambda,\kappa} = +\infty$ uniformly in compact subsets of $\overline{\Omega}_{b,0}\setminus \partial\Omega$
which proves \eqref{k1}. Conversely, $M_{\lambda_{b,0}} \leq \liminf_{\kappa\downarrow 0}\Theta_{\lambda,\kappa}$ in $\overline{\Omega}_+$, where $M_{\lambda_{b,0}}$ stands for the minimal positive solution of \eqref{Ploglargo} with $\lambda = \lambda_{b,0}$, since $\lim_{\varepsilon \rightarrow 0^+} \Theta_{\lambda_{b,0}-\varepsilon}= M_{\lambda_{b,0}}$ in $\overline{\Omega}_+$. In particular, $
\lim_{\kappa \downarrow 0} \Theta_{\lambda,\kappa} = \infty$ on $\partial\Omega_+$.
By a rather standard compactness argument combined with Proposition~\ref{cotakappa} (see for
instance \cite[Proposition 3.3]{Metas}), we obtain that the point-wise limit
$$
M_\lambda(x):=\lim_{\kappa\downarrow 0} \Theta_{\lambda,\kappa}(x)
$$
provide us a classical positive solution of \eqref{Ploglargo} and this finalizes the proof.
\end{demT2}


\end{document}